\documentclass[11pt]{amsart}
\usepackage{amssymb}
\newtheorem{theorem}{Theorem}[section]

\theoremstyle{definition}
\newtheorem{definition}[theorem]{Definition}

\theoremstyle{remark}

\numberwithin{equation}{section}

\newcommand{\CL}{\mathcal{CL}}
\newcommand{\MWP}{\mathcal{MWP}}
\newcommand{\CB}{\mathcal{CB}}
\newcommand{\MT}{\mathcal{M}_T}
\newcommand{\hh}{\mathcal{H}}

\newcommand{\PP}{\mathcal{P}}
\begin{document}

\title{On the $(s,r)$-contractive multi-valued operator and affirmative answers to an open question}

\author{Farshid Khojasteh}
\address{Young Researcher and Elite Club, Arak Branch, Islamic Azad University, Arak, Iran.}
\curraddr{Department of Mathematics, Arak Branch, Islamic Azad University, Arak, Iran.}
\email{f-khojaste@iau-arak.ac.ir,fr\_khojasteh@yahoo.com}
\thanks{The author would like to thanks of Young Researcher and Elite Club which have supported the research.}

%\author{Author Two}
%\address{Mathematical Research Section, School of Mathematical Sciences,
%Australian National University, Canberra ACT 2601, Australia}

%    General info
\subjclass[2010]{Primary 47H10}

\date{November 08, 2018}

\dedicatory{This paper is dedicated to Professor Ovidiu Popescu who suggested these worthwhile results}

\keywords{$(s,r)-$contractive multi-valued operator, Fixed point, Strict fixed point, Nadler fixed point, Pompeiu-Hausdorff metric}

\begin{abstract}
The current research is devoted to find an affirmative answer to an open question related to $(s,r)-$contractive operators which have been introduced by Ovidiu Popescu.
\end{abstract}

\maketitle
\section{Intoduction}
Let $(X,d)$ be a metric space, $\PP(X)$ the set of all nonempty subset of $X$, $\CB(X)$ denote the
classes of all non empty, closed and bounded subsets of $X$, and let $\CL(X)$ the set of all closed subsets of $X$. Let
$T:X \to \CB(X)$ be a multi-valued mapping on $X$. A point
$x\in X$ is called a fixed point of $T$ if $x\in Tx$. Let
$Fix(T)=\{x\in X:x\in Tx\}$. An element $x\in X$ is said to be an
{\bf strict fixed point}(endpoint) of $T$, if $Tx=\{x\}$. Let
$Fix(T)=\{x\in X:x\in Tx\}$(see \cite{1} for more detail).

Let $\hh$ be the Pompeiu-Hausdorff \cite{pamp} metric on $\CB(X)$ induced by $d$,
that is,
\begin{equation}\label{d3}
\hh(A,B):=\max\Big{\{}\sup_{x\in B}d(x,A),\sup_{x\in A}d(x,B)\Big{\}},
\quad A,B\in \CB(X).
\end{equation}
in which, $d(x,A)=\inf\{d(x,y):y\in A\}$

In 2007, Suzuki \cite{7} introduced the following theorem as a simple generalization of Banach contraction principle.
\begin{theorem}\label{Theorem1.2}
Let $(X,d)$ be a complete metric space and let $T$ be a mapping on $X$. Define a non increasing function $\theta$ from $[0,1)$ into $(\frac{1}{2},1]$ by
\[
\theta(r)=\left\{
\begin{array}{lll}
1 &&0\leqslant r <\frac{\sqrt{5}-1}{2}\\
\frac{1-r}{r^2} &&\frac{\sqrt{5}-1}{2}\leqslant r<\frac{1}{\sqrt{2}}\\
\frac{1}{1+r} && \frac{1}{\sqrt{2}}\leqslant r<1
\end{array}
\right.
\]
Assume that there exists $r\in[0,1)$ such that
\begin{equation}\label{eq2}
\theta(r)d(x,Tx)\leqslant d(x,y) \ \ \ \ implies \ \ \ \ d(Tx,Ty)\leqslant r d(x,y)
\end{equation}
for all $x,y\in X$. Then there exists a unique fixed point $z$ of $T$. Moreover, $\lim\limits_{n\to\infty}T^n(x)=z$, for all $x\in X$.
\end{theorem}
\begin{definition}
  Let $(X,d)$ be a metric space and $T:X\to\CL(X)$ a multi-valued operator. $T$ is called Multi-valued Weakly Picard operator ($\MWP$) if for all $x\in X$ and $y\in Tx$, there exists a sequence $\{x_n\}_{n}$ such that:
  \begin{itemize}
    \item[$(a_1)$] $x_0=x$, \ $x_1=y$,
    \item[$(a_2)$] $x_{n+1}\in Tx_n$, for all $n\geqslant 0$,
    \item[$(a_3)$] the sequence $\{x_n\}_n$ is convergent and its limit is a fixed point of $T$.
  \end{itemize}
\end{definition}
In 2013, Popescu \cite{popescu} introduced $(s,r)-$contractive multi-valued operators as follows:
\begin{definition}\cite[Definition 1.7]{popescu}
  Let $(X,d)$ be a complete metric space and let $T$ be a mapping from $X$ into $\CB(X)$. $T$ is called an $(s,r)-$contractive multi-valued operator if $r\in[0,1)$, $s\geqslant r$ and for every $x,y\in X$
  \[
  D(y,Tx)\leqslant s d(y,x) \ \ \ implies \ \ \ \hh(Tx,Ty)\leqslant r \MT(x,y)
  \]
  in which
  \begin{equation}\label{ewq3}
  \MT(x,y)=\max\left\{d(x,y),D(x,Tx),D(y,Ty),\frac{D(x,Ty)+D(y,Tx)}{2}\right\}
  \end{equation}
\end{definition}
After that, he presented the following results
\begin{definition}\label{d32}\cite[Definition 4.2]{popescu}
Let $(X,d)$ be a metric space, $Y\in P(X)$ and $T:X\to\CL(X)$ be a multi-valued operator. Then, we say that the fixed point problem is well-posed for $T$ with respect to $\hh$ if
\begin{enumerate}
  \item[$(a_1)$]$SFix(T)=\{z\}$;
  \item[$(a_2)$]If $\{x_n\}\in Y$, $n\in\mathbb{N}$ and $\hh(x_n,Tx_n)\to 0$ as $n\to\infty$, then $d(x_n,z)\to0$ as $n\to\infty$.
\end{enumerate}
\end{definition}
\begin{theorem}\label{2.3}\cite[Theorem 4.4]{popescu}
Let $(X,d)$ be a complete metric space and $T:X\to\CL(X)$ be a multi-valued operator. We suppose that:
\begin{enumerate}
  \item[$(i)$] $T$ is an $(s,r)-$contractive multi valued operator with $s\geqslant 1$;
  \item[$(ii)$] $SFix(T)\neq\emptyset$.
\end{enumerate}
Then:
\begin{enumerate}
  \item[$(a)$] $Fix(T)=SFix(T)=\{z\}$,
  \item[$(b)$] the fixed point problem is well-posed for $T$ with respect to $\hh$ if $s>1$.
\end{enumerate}
\end{theorem}
%\begin{theorem}\cite[Theorem 2.7]{popescu}\label{2.2}
%Let $(X,d)$ be a complete metric space and $T:X\to\CB(X)$ be a multi-valued operator. Assume that there exists $r,s\in[0,1)$, $r<s$ such that
%\[
%\frac{1}{1+r}D(x,Tx)\leqslant d(x,y)\leqslant\frac{1}{1-s} \ \ \ implies \ \ \ \hh(Tx,Ty)\leqslant r\MT(x,y),
%\]
%in which $\MT(x,y)$ be defined as (\ref{ewq3}). Then $T$ is an $\MWP$ operator.
%\end{theorem}
In Popescu's paper, the Theorem \ref{2.3} gave rise to the following open question:
\begin{theorem}(Open Question)\label{open1}
Let $(X,d)$ be a complete metric space and $T:X\to\CL(X)$ be a multi-valued operator. We suppose that:
\begin{enumerate}
  \item[$(i)$] $T$ is an $(1,r)-$contractive multi valued operator;
  \item[$(ii)$] $SFix(T)\neq\emptyset$.
\end{enumerate}
Then is the fixed point problem is well-posed for $T$ with respect to $\hh$ if $s=1$?
\end{theorem}
\section{Main Results}
In this section, we want to give affirmative answer to the Open Question \ref{open1}.
\begin{theorem}(Answer to Open Question \ref{open1})\label{3.3}
Let $(X,d)$ be a complete metric space and $T:X\to\CL(X)$ be a multi-valued operator. We suppose that:
\begin{enumerate}
  \item[$(i)$] $T$ is an $(1,r)-$contractive multi valued operator;
  \item[$(ii)$] $SFix(T)\neq\emptyset$.
\end{enumerate}
Then:
\begin{enumerate}
  \item[$(a)$] $Fix(T)=SFix(T)=\{z\}$,
  \item[$(b)$] the fixed point problem is well-posed for $T$ with respect to $\hh$.
\end{enumerate}
\end{theorem}
\begin{proof}
Proving $(a)$ is analogous to the argument indicated in \cite[Theorem 4.4(a)]{popescu} and we remain it to the reader. To prove $(b)$, suppose that
$\{x_n\}\in Y$, $n\in\mathbb{N}$ and $\hh(x_n,Tx_n)\to 0$ as $n\to\infty$. We claim that $d(x_n,z)\to0$ as $n\to\infty$.

Since $D(x_n,Tz)=d(x_n,z)\leqslant d(x_n,z),(s=1)$, we have
\begin{equation}\label{4.1}
\hh(Tx_n,\{z\})=\hh(Tx_n,Tz)\leqslant r \MT(x_n,z).
\end{equation}
Also,
\begin{equation}\label{4.2}
  \begin{array}{lll}
    {\displaystyle \MT(x_n,z)}&=&{\displaystyle\max\left\{d(x_n,z),D(x_n,Tx_n),\frac{D(x_n,Tz)+D(z,Tx_n)}{2}\right\}} \\
     & \leqslant & {\displaystyle\max\left\{d(x_n,z),\hh(\{x_n\},Tx_n),\frac{d(x_n,z)+\hh(\{z\},Tx_n)}{2}\right\}}\\
     &\leqslant &  {\displaystyle\max\left\{d(x_n,z),\hh(\{x_n\},Tx_n),\frac{2d(x_n,z)+\hh(\{x_n\},Tx_n)}{2}\right\}}\\
     &=& {\displaystyle\max\left\{d(x_n,z),\hh(\{x_n\},Tx_n),d(x_n,z)+\hh(\{z\},Tx_n)\right\}}
  \end{array}
\end{equation}
(\ref{4.1}) and (\ref{4.2}) conclude that
\begin{equation}\label{4.3}
  {\displaystyle\hh(Tx_n,\{z\})\leqslant r (d(x_n,z)+\hh(\{x_n\},Tx_n))\leqslant r (\hh(\{z\},Tx_n)+2\hh(\{x_n\},Tx_n))}
\end{equation}
Therefore, (\ref{4.3}) concludes that
\[
\hh(Tx_n,\{z\})\leqslant \frac{2r}{1-r}\hh(\{x_n\},Tx_n).
\]
and so we have
\begin{equation}\label{4.6}
  \lim_{n\to\infty}\hh(Tx_n,\{z\})=0.
\end{equation}
Moreover,
\begin{equation}\label{4.7}
  d(x_n,z)\leqslant\hh(\{x_n\},Tx_n)+\hh(Tx_n,\{z\})\leqslant \frac{1+r}{1-r}\hh(\{x_n\},Tx_n).
\end{equation}
Applying (\ref{4.6}) and taking limit on both sides of (\ref{4.7}) and the fact that $\hh(\{x_n\},Tx_n)\to 0$ as $n\to\infty$, one can obtain desired result and the proof is completed.
\end{proof}

%\begin{theorem}\label{3.2}
%Let $(X,d)$ be a complete metric space and $T:X\to\CB(X)$ be a multi-valued operator. Assume that there exists $r,s\in[0,1)$, $r\leqslant s$ such that
%\[
%\frac{1}{1+r}D(x,Tx)\leqslant d(x,y)\leqslant\frac{1}{1-s} \ \ \ implies \ \ \ \hh(Tx,Ty)\leqslant r\MT(x,y),
%\]
%in which $\MT(x,y)$ be defined as (\ref{ewq3}). Then $T$ is an $\MWP$ operator.
%\end{theorem}
%\begin{proof}
%
%\end{proof}

\bibliographystyle{amsplain}

\end{document}